\documentclass[11pt,a4paper,hyperref]{amsart}
\usepackage[utf8]{inputenc}
\usepackage[T1]{fontenc}
\usepackage{textcomp}
\usepackage{lmodern}
\usepackage{hyperref}

\newtheorem{thm}{Theorem}[section]
\newtheorem{cor}[thm]{Corollary}
\newtheorem{lem}[thm]{Lemma}
\newtheorem{prop}[thm]{Proposition}

\newtheorem{eks}{\sc Example}
\theoremstyle{definition}
\newtheorem{defn}[thm]{Definition}
\theoremstyle{remark}
\newtheorem{rem}{Remark}[section]
\numberwithin{equation}{section}
\newcommand{\eps}{\varepsilon}

\newcommand{\Ss}{S^{\ast}}

\newcommand{\xs}{x^\ast}
\newcommand{\Xs}{X^{\ast}}
\newcommand{\Xss}{{X^{\ast\ast}}}

\newcommand{\Xsss}{X^{\ast\ast\ast}}

\newcommand{\Ys}{Y^{\ast}}
\newcommand{\ys}{y^\ast}

\newcommand{\PLR}{Principle of Local Reflexivity}

\newcommand{\HB}{\text{H{\kern -0.35em}B}}

\DeclareMathOperator{\spann}{span}

\DeclareMathOperator{\linspan}{span}

\begin{document}

\title[Almost isometric ideals in Banach spaces]{Almost isometric
  ideals in Banach spaces}

\author[T.~A.~Abrahamsen]{Trond~A.~Abrahamsen}
\address{Department of Mathematics, Agder University, Servicebox 422,
4604 Kristiansand, Norway.} \email{Trond.A.Abrahamsen@uia.no}
\urladdr{http://home.uia.no/trondaa/index.php3}

\author[V.~Lima]{Vegard Lima}
\address{{\AA}lesund University College, Postboks 1517, 6025
  {\AA}lesund, Norway.} \email{Vegard.Lima@gmail.com}

\author[O.~Nygaard]{Olav Nygaard}
\address{Department of Mathematics, Agder University, Servicebox 422,
4604 Kristiansand, Norway.} \email{Olav.Nygaard@uia.no}
\urladdr{http://home.uia.no/olavn/}


\begin{abstract} A natural class of ideals, almost isometric ideals, of Banach
  spaces is defined and studied. The motivation for working with this
  class of subspaces is our observation that they inherit diameter 2
  properties and the Daugavet property. Lindenstrauss spaces are known
  to be the class of Banach spaces that are ideals in every
  superspace; we show that being an almost isometric ideal in every superspace
  characterizes the class of Gurariy spaces.
\end{abstract}

\maketitle
\section{Introduction}\label{sec1}

Let $Y$ be a real Banach space and $X$ a subspace.
Recall that $X$ is an \emph{ideal} in $Y$ if $X^\perp$,
the annihilator of $X$ in $Y^*$, is the kernel
of a norm one projection on $Y^*$.
A linear operator $\varphi: X^* \to Y^*$ is called
a \emph{Hahn-Banach extension operator} if
$\varphi(x^*)(x) = x^*(x)$ and $\|\varphi(x^*)\|=\|x^*\|$
for all $x \in X$ and $x^* \in X^*$.
If for every finite dimensional subspace $E$ of $Y$
and every $\varepsilon > 0$ there exists a linear
operator $T: E \to X$ such that
$Te = e$ for all $e \in E \cap X$ and
$\|T\| \le 1 + \varepsilon$
then $X$ is said to be \emph{locally 1-complemented} in $Y$.
That these three concepts is just the same thing
looked at in three different ways dates back
to a 1972-paper of Fakhoury.
We will mainly use the locally 1-complemented viewpoint, but following
the paper \cite{GKS} we will use the term ideal.
The next theorem can be found in
\cite[Th\'eor\`eme~2.14]{Fak} (see also \cite[Theorem~3.5]{Kal}).

\begin{thm}\label{thm:fakhourykalton}
  Let $X$ be a subspace of a Banach space $Y$.
  The following statements are equivalent.
  \begin{itemize}
  \item[(i)]
    $X$ is an ideal in $Y$.
  \item[(ii)]
    There exists a Hahn-Banach extension operator
    $\varphi: X^* \to Y^*$.
  \item[(iii)]
    $X$ is locally 1-complemented in $Y$.
  \end{itemize}
\end{thm}

The connection between the extension operators
and the locally complemented subspaces was
further explored in \cite{MR2262909}.
There the following theorem can be found. 

\begin{thm}\label{thm:intro_thm-ideal}
  Let $X$ be a subspace of a Banach space $Y$.
  $X$ is an ideal in $Y$ if and only if
  there exists a Hahn-Banach extension operator
  $\varphi:X^* \to Y^*$ such that for every $\varepsilon>0$,
  every finite-dimensional subspace $E\subset Y$ and every
  finite-dimensional subspace $F\subset X^*$
  there exists $T:E\to X$ such that
  \begin{itemize}
  \item [(i1)] $Te=e$ for all $e\in X\cap E$,
  \item [(i2)] $\|Te\|\leq (1+\varepsilon)\|e\|$ for all $e\in E$, and
  \item [(i3)] $\varphi f^*(e)=f^*(Te)$ for all $e\in E,
    f^*\in F$.
  \end{itemize} 
\end{thm}

Let us now describe the content of our paper: In searching for a
natural condition ensuring that $X$ inherits the property from its
superspace $Y$ that every non-void relatively weakly open subset of
$B_Y$ has diameter 2, we observed that if the $T$ in Theorem
\ref{thm:intro_thm-ideal} can be assumed to be an $\eps$-isometry,
then this \em diameter 2 property \em passes down to $X$ from
$Y$. This observation is presented in Proposition
\ref{prop:diam2subspace}. Also, we observed that the same condition
works for the problem of inheriting the Daugavet property. The
presentation of this result can be found in Proposition
\ref{prop:daugavet}. Precise definitions and necessary background on
both diameter 2 properties and the Daugavet property are incorporated
in the presentation in Section \ref{sec3}.

The above results on inheriting the diameter 2 property and the
Daugavet property indicate that subspaces obeying the conclusion in
Theorem \ref{thm:intro_thm-ideal} with $T$ almost isometric are of
some relevance, and we think it is natural to find out what can be
said in general about such subspaces: 

\begin{defn}\label{def:super1comp}
  Let $Y$ be a Banach space and $X$ a subspace.
  $X$ is called an \em almost isometric ideal (ai-ideal) \em in $Y$
  if  for every $\eps>0$
  and every finite-dimensional subspace $E\subset Y$ there exists
  $T:E\to X$ which satisfies the condition (i1) in Theorem
  \ref{thm:intro_thm-ideal} and also
  \begin{itemize}
  \item [(ai2)] $(1+\eps)^{-1}\|e\|\leq\|Te\|\leq (1+\eps)\|e\|$ for
    $e\in E$.
\end{itemize}
\end{defn}

Note that $X$ need not be closed in the definition of (ai-)ideals.
By a perturbation argument a non-closed subspace is an (ai-)ideal
if and only if its closure is.

\begin{rem}\label{rem:superGEfdr}
  A Banach space $Y$ is \emph{finitely representable} in $X$ if
  for every finite-dimensional subspace $E\subset Y$
  there exists a $T: E \to X$ such that (ai2) holds.
  
  There is a 1-complemented isometric copy of
  $\ell_1$ in $L_1[0,1]$ (see e.g. \cite[Lemma 5.1.1]{AlKal}) and in particular
  $\ell_1$ is an ideal.
  $L_1[0,1]$ is finitely representable in
  $\ell_1$, but $\ell_1$ is not an ai-ideal in $L_1[0,1]$
  because ai-ideals inherit diameter 2 properties
  (see Proposition~\ref{prop:diam2subspace} below).
  Hence there is no $T$ that satisfies
  properties (i1) and (ai2) simultaneously.
\end{rem}

A natural question is whether the analogue of Theorem
\ref{thm:intro_thm-ideal} holds. Lindenstrauss' compactness argument
of course produces a Hahn-Banach extension operator, but the problem
is that we risk losing the $\eps$-isometry property of
$T$. It turns out that the analogue of Theorem
\ref{thm:intro_thm-ideal} is true:

\begin{thm}\label{thm:intro_thm-superideal} Assume that $X$ is an
  ai-ideal in $Y$. Then there exists a Hahn-Banach operator
  $\varphi:\Xs\to\Ys$ such that for every $\eps>0$, every
  finite-dimensional subspace $E\subset Y$ and every
  finite-dimensional subspace $F\subset \Xs$ there exists $T:E\to X$
  which satisfies the statements (i1), (ai2), and (i3).
\end{thm}

The proof of this structure result will be the starting point of
Section \ref{sec2}.   

Note that the conclusion of Theorem \ref{thm:intro_thm-superideal} is
very similar to the {\PLR}, and a Banach space $X$ is always an almost
isometric ideal
in its bidual $\Xss$. By Goldstine's theorem, in the {\PLR} setting,
the range of $\varphi:\Xs\to\Xsss$ is 1-norming for $\Xss$ (here
$\varphi$ is simply the canonical embedding of $\Xs$ into
$\Xsss$). We will see in
Proposition \ref{prop:strong-super} that it is in general true that
when the range of $\varphi$ in Theorem \ref{thm:intro_thm-ideal} is
1-norming for $Y$, then the ideal $X$ is an ai-ideal in $Y$. 

Knowing this, our next question is naturally whether for an almost
isometric ideal
the associated Hahn-Banach extension operators $\varphi$ from $\Xs$
into $\Ys$ must have range that is 1-norming for $Y$. We will see that
this is not so in general; in Example \ref{ex:not} we will see that
the 1-co-dimensional subspace $X = \{(a_n)_{n=1}^\infty \in c_0:
a_1=0\}$ of $c_0$ is a counterexample.

However, in the important case of u-ideals the $\eps$-isometry
condition of T and having 1-norming range are indeed equivalent
(Theorem \ref{thm:super-strict-u-ideal}). U-ideals were introduced and studied
in \cite{GKS}; we also give some necessary background on u-ideals in
the introduction to our Theorem \ref{thm:super-strict-u-ideal}.

To sum up, our motivation is the question of when diameter 2 properties and
the Daugavet property pass to subspaces. This leads to the concept of an
ai-ideal. Almost isometric ideals are studied in Section
\ref{sec2}, and the
results on the diameter 2 properties and the Daugavet property form
Section \ref{sec3}. 

In Section \ref{sec4} we characterize Gurariy
spaces in terms of ai-ideals: From \cite{Fak} it is known that the
Banach spaces that form an ideal in every superspace is exactly the
class of Lindenstrauss spaces. We observe in Theorem \ref{thm:gurarii}
that the class of spaces that are ai-ideals in every superspace is
the Gurariy spaces. From this it follows that Gurariy spaces
have the Daugavet property. We end the paper by proving that
Lindenstrauss spaces in general enjoy the diameter 2 properties.

We use standard Banach space notation; symbols and terms will however
be carefully explained throughout the text when we think it is helpful
to the reader. The reader only interested in the results on the
passage of diameter 2 properties or the Daugavet property to subspaces
may go directly to Section \ref{sec3}.

\section{Ai-ideals, strict ideals , and u-ideals}\label{sec2}

We start by proving our main structure theorem
which was stated in the introduction.

\begin{proof}[Proof of Theorem~\ref{thm:intro_thm-superideal}]
  We first construct $\varphi$ using a Lindenstrauss compactness
  argument.
  Order the set
  $A = \{(E,F,\varepsilon)\}$,
  where $E \subset Y$ and $F \subset X^*$ are finite-dimensional
  and $\varepsilon > 0$, by
  $(E_1,F_1,\varepsilon_1) \le (E_2,F_2,\varepsilon_2)$
  if $E_1 \subset E_2$, $F_1 \subset F_2$, and
  $\varepsilon_2 \le \varepsilon_1$.

  For $\alpha \in A$, $\alpha=(E,F,\varepsilon)$,
  choose $T_\alpha : E \to X$ satisfying (i1) and (ai2).
  Define
  $L_\alpha:Y \to X^{**}$ by $L_\alpha y=T_\alpha y$ if $y \in E$ and
  $L_\alpha y=0$ if $y \notin E$. We consider $(L_\alpha) \subset
  \Pi_{y \in Y} B_{X^{**}}(0,2\|y\|)$ which by Tychonoff
  is compact in the product weak$^*$ topology.
  Without loss of generality we assume
  $(L_\alpha)$ is convergent to some
  $S \in \Pi_{y \in Y} B_{X^{**}}(0,2\|y\|)$.
  Note that this implies that for
  every finite number of elements $(y_i)_{i=1}^n $ in $Y$ and
  $(x^{*}_j)_{j=1}^m$ in $X^*$ we have 
  \begin{equation}\label{eq:1}
    x^{*}_j(L_\alpha y_i) \to x^{*}_j(Sy_i).
  \end{equation}
  By construction $Sy=y$ for every $y
  \in X$. It is also clear that $\|S\|=1$, hence
  $\varphi = S^*|_{X^*} : X^* \to Y^*$ is a
  Hahn-Banach extension operator.
  
  Next we apply a perturbation argument modelled after \cite[Lemma
  1.2]{MR2262909} which in turn was inspired by \cite{MR0280983}.
   
  Let $(x_i^*,x_i)_{i=1}^n$ be a complete biorthogonal system
  for $F$.  Define
  \begin{equation*}
    Q = \sum_{i=1}^n i_X x_i \otimes \varphi(x_i^*).
  \end{equation*}
  Here $i_X: X \to Y$ is the identity embedding.
  Then $Q \in \mathcal{F}(Y^*,Y^*)$ is a
  projection with $Q(Y^*) = \varphi(F)$,
  and $Q^*(Y^{**}) \subset X$.
  Similarly let $P \in \mathcal{F}(E,E)$
  be a projection with $P(E) = E \cap X$.
  
  For $\alpha \in A$, $\alpha = (E,F,\varepsilon)$, let
  $(T_\alpha)$ be the net from the first paragraph.
  Define $S_\alpha : E \to X$ by
  \begin{align*}
    S_\alpha &=
    i_EP + T_\alpha(I_E - P) - Q^*(T_\alpha - i_E)(I_E-P) \\
    &=
    i_E + (I_{Y^{**}} - Q^*)(T_\alpha - i_E)(I_E-P).
  \end{align*}
  Here $i_E:E \to Y$ denotes the identity embedding. Now $S_\alpha \in
  \mathcal F(E,X)$, because $i_EP(E) = E \cap X \subset X$ and
  $Q^*(Y^{**}) \subset X$ and $P, T_\alpha$, and $i_E$ are
  finite-rank operators. We have $S_\alpha e = e$ for every $e \in E
  \cap X$ because $E \cap X = P(E)$ and $P$ is a projection.

  Let $f^* \in F$ and $e \in E$. Using $S_\alpha(E) \subset X$ we
  have
  \begin{align*}
    \langle f^*, S_\alpha e \rangle
    &=
    \langle \varphi f^*, S_\alpha e \rangle \\
    &=
    \langle \varphi f^*, i_E e \rangle
    +
    \langle \varphi f^*, (I_{Y^{**}}-Q^*)(T_\alpha-i_E)(I_E-P) e \rangle \\
    &=
    \langle \varphi f^*, i_E e \rangle
    +
    \langle (I_{Y^{*}}-Q)\varphi f^*, (T_\alpha-i_E)(I_E-P) e \rangle \\
    &=
    \langle \varphi f^*, i_E e \rangle
  \end{align*}
  since $Q(\varphi f^*) = \varphi f^*$.

  So far we have shown that $(S_\alpha)$ satisfies
  (i1) and (i3). Far out in the net the $S_\alpha$'s
  will inherit (ai2) from the $T_\alpha$'s if
  we can show that $\|S_\alpha - T_\alpha\|$
  can be made a small as we wish.
  Note that
  \begin{equation*}
    S_\alpha - T_\alpha =
    (i_E - T_\alpha)P - Q^*(T_\alpha - i_E)(I_E - P)
    = - Q^*(T_\alpha - i_E)(I_E - P)
  \end{equation*}
  since $T_\alpha e = e$ for all $e \in P(E)$.
  Thus we have
  \begin{equation*}
    \|S_\alpha - T_\alpha\| = \sup_{\|e\| = 1}\|Q^*(T_\alpha - i_E)e\|
    \le
    \sup_{\|e\| = 1} \sum_{i=1}^n \|x_i\|
    |x^*_i(T_\alpha e) - \varphi x^*_i(e)|.
  \end{equation*}

  Let $\alpha = (E,F,\varepsilon)$.
  Let $\delta > 0$ and choose a $\delta$-net
  $(e_j)_{j=1}^k$ for $S_E$.
  We choose $\beta \ge \alpha$ so that
  $|x^*_i (T_\beta e_j) - \varphi x^*_i(e_j)| < \delta$
  for every $i = 1, \ldots, n$ and $j = 1,\ldots,k$
  using \eqref{eq:1}.

  For $e \in S_E$ choose $j$ such that $\|e - e_j\| < \delta$,
  then
  \begin{align*}
    |x_i^*(T_\beta e) - \varphi x^*_i(e)| &\le
    |x_i^*(T_\beta e) - x_i^*(T_\beta e_j)|
    + |x_i^*(T_\beta e_j) - \varphi x^*_i(e_j)| \\
    & + |\varphi x^*_i(e_j) - \varphi x^*_i(e)|\\
    &\le 2\|x^*_i\|\delta + \delta
    + \|x^*_i\|\delta
    \le \delta(1+3\max_i\|x_i^*\|).
  \end{align*}
  By choosing $\delta$ small enough we get that
  $S_\beta$ satisfies (i1), (ai2) and (i3) for the
  given $\alpha = (E,F,\varepsilon)$.
  The desired $T: E \rightarrow X$ is then $S_\beta$.
\end{proof}

In the proof above we used the connection $\varphi=\Ss|_{\Xs}$
between a Hahn-Banach extension operator
$\varphi: \Xs \to \Ys$ and a norm one extension $S:Y\to\Xss$ of the
canonical embedding $k_X:X\to\Xss$. Clearly, from this connection the
existence of a Hahn-Banach extension operator and a
norm one extension of $k_X$ to $Y$ are equivalent. Moreover,
the existence of a Hahn-Banach extension operator $\varphi:\Xs\to\Ys$
is equivalent to the existence of a norm one projection $P$ on $\Ys$
with $\ker P=X^{\perp}$ and range equal to $\varphi(\Xs)$.

From the way $P, \varphi$ and $S$ are connected, one obtains
that the range of $\varphi$ (or $P$)
is 1-norming if and only if $S$ is an isometry into. This is
well-studied in the recent literature (see e.g. \cite{rao} or
\cite{LL}); these ideals are called \em strict ideals. \em 

\begin{prop}\label{prop:strong-super} Suppose $X$ is a strict ideal in
  $Y$. Then $X$ is an ai-ideal in $Y$. 
\end{prop}

\begin{proof}
  Let $\eps>0$, $E \subset Y$
  finite-dimensional, and $S:Y \to \Xss$ an isometric extension of
  $k_X$. Since $F = S(E) \subset \Xss$ is finite-dimensional, there
  exists by the {\PLR } $T: F \to X$ satisfying (i1) and (ai2). It is 
  clear that also the composition $T \circ S: E \to X$ satisfies (i1) and
  (ai2) and so we are done.    
\end{proof}

We now give an example which shows that the converse of Proposition
\ref{prop:strong-super} is not true. For
this example we will just need a little more background on ideals. 
An ideal $X\subset Y$ is
an \emph{M-ideal} in $Y$ if the ideal projection
$P:\Ys\to\Ys$ is an L-projection, that is,
\[\|\ys\|=\|P\ys\|+\|\ys-P\ys\|\hspace{1cm}\mbox{for
  all}\hspace{3mm}\ys\in\Ys.\]
A particular case of this is when $X$ is 1-complemented in
$Y$ by an M-projection $Q$, that is $QY=X$ and
\[\|y\|=\max\{\|Qy\|,\|y-Qy\|\}\hspace{1cm}\mbox{for all}\hspace{3mm}y\in Y,\]
in which case $X$ is called an \emph{M-summand} in $Y$. 
The M-ideal projection is unique (see
e.g. \cite[Proposition~1.2]{MR735420} or \cite[p. 2]{HWW}). Further,
if $X$
is also an M-summand, then $Q^{\ast}(\Ys)$ is weak$^\ast$ closed, hence
if $X$ is a proper subspace of $Y$ it can not be a strict ideal in $Y$.

We denote by $e_n$ the n-th standard basis vector in $c_0$ and by
$e_n^\ast$ its corresponding coordinate functional in $\ell_1$.

\begin{eks}\label{ex:not}
  The subspace $X = \{(a_n)_{n=1}^\infty \in c_0: a_1=0\}=\ker
  e_1^\ast$ of $c_0$ is  1-complemented and an ai-ideal in $c_0$.
\end{eks}

\begin{proof}
Clearly $X$ is a proper M-summand in $c_0$ complemented by the projection $Q$
putting 0 on the first coordinate, and by the above remarks we only
need to show that $X$ is an ai-ideal. 

Let $E$ be a finite-dimensional subspace of $c_0$ and let
$(x_i)_{i=1}^m$ be some $\eps$-net for $S_E$.
Find $N$ such that $|x_i(N)| < \eps$ for $i=1,2,...,m.$ Define $T: E
\to X$ by $T(y) = Qy + e_1^*(y)e_N.$
Then $T$ is obviously linear and an $\eps$-isometry on $(x_i)_{i=1}^m.$
By  \cite[Lemma 11.1.11]{AlKal}  $T$ is an almost isometry on all of $E$.
\end{proof}

As we have seen from Example \ref{ex:not} ai-ideals need not be
strict. We will now show that if some symmetry condition is imposed,
then ai-ideals indeed are strict. A subspace $X$ is said
to be a \emph{u-ideal} in $Y$ if there exists an ideal projection
$P:\Ys\to\Ys$ such that $\|I-2P\|=1$ ($P$ is \emph{unconditional}). If
the range of $P$ is 1-norming for $Y$, then $X$ is called a
\emph{strict u-ideal}
in $Y$. There can never be more than one unconditional $P$ (\cite[Lemma
3.1]{GKS}). Further, every M-ideal is a u-ideal. From \cite[Proposition
3.6]{GKS} it is known that $X$ is a u-ideal if and only if $X$ is an
ideal with the extra condition
  \begin{enumerate}
  \item[(i4)]\label{item:suid4}
    $\|e - 2T(e)\| \le (1+\varepsilon)\|e\|$ for all $e \in E$.
  \end{enumerate}

We will now assume that $X$ is a u-ideal in $Y$ and that the $T$'s
above can be chosen to be almost isometries:

\begin{defn}\label{defn:aiu-ideal}
   A subspace $X$ is called an \emph{almost isometric u-ideal}
   (ai-u-ideal) in $Y$ if for every
   $\eps>0$ and every finite-dimensional subspace $E\subset Y$ there
   exists $T:E\to X$ which satisfies the conditions (i1) and (ai2) and
   also (i4).
\end{defn}

\begin{rem} \label{rem:ai-u-ideal}
  Note that (i2) (and the right-hand inequality of
  (ai2)) follows from (i4). 

  An inspection of the proof of
  Theorem \ref{thm:intro_thm-superideal}
  shows that when $X$ is an ai-u-ideal in $Y$, it is possible to
  obtain a Hahn-Banach extension operator $\varphi:\Xs\to\Ys$ wich
  satisfies (i3). However, this observation will not be needed in what
  follows. 
\end{rem}

\begin{thm}\label{thm:super-strict-u-ideal}
  Let $X$ be a subspace of $Y$. Then $X$ is an ai-u-ideal in $Y$ if and
  only if $X$ is a strict u-ideal in $Y$.
\end{thm}

\begin{proof}
  Assume that $X$ is a strict u-ideal in $Y$.
  Let $E \subset Y$ and $F \subset X^*$ be finite dimensional
  subspaces. Let $\phi:X^* \to Y^*$ be the strict unconditional
  extension operator. Let $\mathcal{L}(E,X)$ denote
  the bounded linear operators from $F$ to $X$
  and let $i_E \in \mathcal{L}(E,Y)$ be the identity embedding.

  Define $\Phi : \mathcal{L}(E,X)^* \to \mathcal{L}(E,Y)^*$
  by $\Phi(e \otimes x^*) = e \otimes \phi x^*$
  as in \cite[Proposition~3.6]{GKS}.
  Using \cite[Lemma~2.2]{GKS} we find a net $(T_\alpha)$
  in $\mathcal{L}(F,X)$ converging weak$^*$ to
  $\Phi^*(i_E) \in \mathcal{L}(E,X)^{**}$ such that
  $\limsup_\alpha \|i_E - 2T_\alpha\| \le \|i_E\| = 1$.
  Applying the perturbation argument from \cite[Lemma~1.2]{MR2262909}
  (as in Theorem~1.4) we get at linear operator $T:E \to X$
  satisfying (i1), (ai2), and (i4).
  (For $T$ to become an almost isometry we may have to
  enlarge $F$.)

  Assume that $X$ is an ai-u-ideal in $Y$.
  Choose $y \in Y \setminus X$.
  Then $X$ is an ai-u-ideal in $Z = \linspan(X,\{y\})$.
  Let $z \in S_Z$ and let $E$ be a finite-dimensional subspace of $Z$
  containing $z$.
  Choose $T : E \to X^{**}$ satisfying (i1), (ai2) and (i4).
  We have $(1-\varepsilon) < (1+\varepsilon)^{-1}$ so by (ai2)
  \begin{equation*}
    (1-\varepsilon) \le \|Tz\| \le (1+\varepsilon)
  \end{equation*}
  hence
  \begin{equation*}
    | \|Tz\| - 1 | \le \varepsilon
  \end{equation*}
  Using (i4) we get
  \begin{equation*}
    \|z - 2\frac{Tz}{\|Tz\|}\|
    \le
    \|z - 2Tz\| + 2 \|Tz - \frac{Tz}{\|Tz\|}\|
    \le
    (1+\varepsilon) + 2|1 - \|Tz\| |
    \le 1+3\varepsilon,
  \end{equation*}
  which shows that
  \begin{equation*}
    \inf_{x \in S_X} \|z - 2x\| = 1.
  \end{equation*}
  By Theorem~2.4 in \cite{LL}
  $X$ is a strict u-ideal in $Z$.
  This is true for any $y \in Y$ and so, by
  Proposition~2.1 in \cite{LL},
  $X$ is a strict u-ideal in $Y$.
\end{proof}


\section{Ai-ideals inherit diameter 2 properties and the
  Daugavet property}\label{sec3}

Let $X$ be a non-trivial (real) Banach space with unit ball $B_X$.
By a \emph{slice} of $B_X$ we mean a set $S(\xs, \eps)=\{x\in
B_X\::\xs(x)>1-\eps\}$ where $\xs$ is in the unit sphere $S_{\Xs}$ of
$\Xs$ and $\eps > 0$. A \emph{finite convex combination of slices} of $B_X$
is then a set of the form \[S=\sum_{i=1}^{n}\lambda_i
S(\xs_i,\eps_i),\:\:\lambda_i \ge 0,\:\:\sum_{i=1}^{n}\lambda_i=1,\]
where $\xs_i \in S_{\Xs}$ and $\eps_i > 0$ for $i=1,2,\ldots,n$.

The relations between the following three successively stronger
properties were investigated in \cite{ALN}:

\begin{defn} \label{defn:diam2p}A Banach space $X$ has the
  \begin{enumerate}
  \item [(i)] {\it local diameter 2 property} if every slice of $B_X$ has
    diameter 2.
  \item [(ii)] {\it diameter 2 property} if every
    non-empty relatively weakly open subset of $B_X$ has diameter 2.
  \item [(iii)] {\it strong diameter 2 property} if every finite convex
    combination of slices of $B_X$ has diameter 2.
  \end{enumerate}
\end{defn}

It is not known to us whether properties (i) and (ii) really are
different. However, in the
recent paper \cite{ABL} it was established that (iii) is strictly
stronger than (ii). (The same result has independently also been discovered
by Haller, Langemets, and P{\~o}ldvere \cite{HLP}). The study of
property (ii) above goes back to Shvidkoy's work \cite{Shv} on the
Daugavet property, where a by-product is that spaces with the Daugavet
property enjoy the diameter 2 property, and to Nygaard and Werner's
paper \cite{NW} where uniform algebras are shown to have the diameter
2 property. We postpone the definition of the Daugavet property until
needed; it can be found in the introduction to Proposition
\ref{prop:daugavet} below. A \emph{uniform algebra} is a separating closed
subalgebra of a $C(K)$-space that contains the constants.  

In addition to Daugavet spaces and uniform algebras, spaces with
``big'' centralizer are also known to have the diameter 2
property. Precise definition of ``big'' centralizer can be found in
\cite{ALN} or \cite{ABG}, we will not give it here as we will not really
need it; for our purposes it is enough to know that Daugavet spaces,
uniform algebras and spaces with ``big'' centralizer form three large
classes of spaces with the diameter 2 property. 

We believe it is folklore among researchers working on the diameter 2
property that $X$ inherits the diameter 2 property from its bidual
$\Xss$, although we do not know any explicit reference for it. Here we
will show the much more general result that all the diameter 2
properties are inherited by ai-ideals.

\begin{prop}\label{prop:diam2subspace}
  Let $X$ be an ai-ideal in a Banach space $Y$. If $Y$ has the
  diameter $2$ property, then so does $X$.
\end{prop}

\begin{proof} 
  Let $\varphi$ be the associated Hahn-Banach extension operator
  from Theorem~\ref{thm:intro_thm-superideal}.
  Let $U\subset B_X$ be relatively weakly open and $\varepsilon>0$.
  We will show that for every $x_0 \in U$
  any set of the form
  \begin{equation*}
    U_\delta=\{x\in B_X\::|x^*_i(x-x_0)|<\delta, i=1,2,...,n\}
  \end{equation*}
  contains two points with distance greater than $2-\varepsilon$.
  Let
  \begin{equation*}
    V_\delta=\{y\in B_Y\::|\varphi x^*_i(y-x_0)|<\delta, i=1,2,...,n\}.
  \end{equation*}
  $V_\delta$ is relatively weakly open in $B_Y$ and
  hence has diameter 2.
  Thus we can find $z_1, z_2 \in V_\delta$
  with $\|z_1\|\le 1$, $\|z_2\|\le 1$ and $\|z_1-z_2\| > 2-\varepsilon/4$.
  Let $0 < \eta < \varepsilon/8$ and set
  $y_i = (1+\eta)^{-1} z_i$. Then $\|y_1-y_2\| > 2-\varepsilon/2$.

  Let $E = \spann\{x_0,y_1,y_2\}$ and let $F=\spann\{x^*_i\}_{i=1}^n$.
  Use Theorem~\ref{thm:intro_thm-superideal} to find
  an $\eta$-isometry $T: E \to X$.
  Then
  $\|Ty_i\| \le (1+\eta)\|y_i\| \le 1$,
  \begin{equation*}
    \|Ty_1-Ty_2\| \ge (1+\eta)^{-1}\|y_1-y_2\| > 2-\varepsilon,
  \end{equation*}
  and for $i=1,2,\ldots,n$ and $j=1,2$ we have
  \begin{equation*}
    |x^*_i(Ty_j-x_0)|=|x^*_i(Ty_j-Tx_0)|=|x^*_i(T(y_j-x_0))|=|\varphi x^*_i
    (y_j-x_0)|<\delta,
  \end{equation*}
  hence $Ty_1, Ty_2\in U_\delta$.
\end{proof}

\begin{rem}
  Note that in the proof above we only needed to be able to push every
  three-dimensional $E\subset Y$ into $X$ almost isometrically.
\end{rem}

Now we prove that also the strong diameter 2 property is inherited by
ai-ideals.

\begin{prop}\label{prop:strongdiam2subspace}
  Let $X$ be an ai-ideal in Banach space $Y$. If $Y$ has the strong
  diameter $2$ property, then so does $X$. 
\end{prop}

\begin{proof}
  Let $S \subset B_X$ be a finite convex combination of slices. $S$ is
  then of the form
    \[S=\sum_{i=1}^n\lambda_iS_i(\xs_i,\varepsilon_i),\]
  where $\xs_i \in B_\Xs, \varepsilon_i>0, \lambda_i>0$, and
  $\sum_{i=1}^n\lambda_i=1$. Now put
    \[S_\varphi=\sum_{i=1}^n\lambda_iS_{\varphi, i}(\varphi\xs_i,\varepsilon_i),\]
  where $\varphi$ is the Hahn-Banach extension operator associated with
  the ai-ideal. Note that each 
  $S_{\varphi, i}(\varphi\xs_i,\varepsilon_i)=\{y \in
  B_Y:\varphi\xs_i(y)>1-\varepsilon_i\}$ is a slice of $B_Y$. Since
  $S_\varphi$ has diameter $2$, there are for every $\eta>0$, $y_k \in
  S_\varphi$, $k=1,2$, such that $\|y_1-y_2\|>2-\eta$. Now $y_k \in
  S_\varphi$ is of the form $y_k=\sum_{i=1}^{n_k}\lambda_i
  y_{\varphi,k}^i$ where $y_{\varphi,k}^i \in
  S_{\varphi, i}(\varphi\xs_i,\varepsilon_i)$. Let
  $E=\spann(y_k,y_{\varphi,k}^{n_k})_{k,i} \subset Y$ and
  $F=\spann(\xs_i)_i \subset \Xs$. By a perturbation argument, we
  may assume that $\max_k\|y_k\|=r<1$.

  For $\delta>0$ such that $(1+\delta)\cdot r\le 1$,
  choose $T:E \to X$ which fulfills (i)-(iii) in the conclusion of
  Theorem \ref{thm:intro_thm-superideal} with this $\delta$, and
  observe that $Ty_k=\sum_{i=1}^{n_k}\lambda_i Ty_{\varphi,k}^i$. Then
  $Ty_k \in S$ since $\|Ty_k\|\le(1+\delta)\|y_k\|\le(1+\delta)\cdot r
  \le 1$, and $Ty_{\varphi,k}^i(\xs_i)=\varphi
  \xs_i(y_{\varphi,k}^i)>1-\varepsilon_i$, so $Ty_{\varphi,k}^i \in
  S_i$.

  Finally, observe that
  $\|Ty_1-Ty_2\|>(1+\delta)^{-1}(2-\eta)$ and that $\delta$ and $\eta$
  may be chosen arbitrarily small, so the diameter of $S$ must be $2$.
\end{proof}

\begin{rem}
  Note that in the proof above we can not take $E$ with just $3$
  dimensions as we could in the proof of the similar result for the
  diameter $2$ property. 
\end{rem}
  
\begin{cor}\label{cor:localdiam2subspace}  Almost isometric ideals inherit the
  local diameter 2 property.
\end{cor}

\begin{proof} Take $n=1$ in the proof of Proposition
  \ref{prop:strongdiam2subspace}.
\end{proof}

\begin{prop}
  Let $Y$ be a Banach space. If every infinite-dimensional separable
  ideal in $Y$ has the (local, strong) diameter 2 property, then so does
  $Y$.
\end{prop}

\begin{proof}
  First let us prove the result for the strong diameter 2
  property. To this end let $\varepsilon_i > 0$ for $i = 1, \ldots,
  n$ and $S = \sum_{i=1}^k \lambda_i S_i$ a finite convex combination of
  slices $S_i = \{y \in B_Y: \ys_i(y) > 1- \eps_i\}$ of the unit ball
  of $Y$. Each slice is
  relatively weakly open and therefore contains a ball of small
  radius about a point in the slice. Thus it is possible to find a
  sequence of infinitely many linearly independent points in each
  slice. It is clearly also possible to find a linearly independent
  sequence $(y_n) \subset S$. Let $Z$ be
  the norm closure of $\mbox{span}(y_n)$. By \cite{MR675426}
  (cf. also \cite[Lemma III.4.3]{HWW}) there is a separable ideal $X$
  in $Y$ containing $Z$ such that $\spann(\ys_i)_{i = 1}^k \subset
  \varphi(\Xs)$ where $\varphi:\Xs \to \Ys$ is the Hahn-Banach
  extension operator. Now, for $i =
  1, \ldots, k$, find $\xs_i \in \Xs$ such that $\ys_i =
  \varphi(\xs_i)$. Let $S_i' = \{x \in B_X: \xs_i(x) > 1- \eps_i\} =
  \{x \in B_X: \varphi \xs_i(x) > 1- \eps_i\}$ be slices of
  the unit ball of $X$. Denote by $S' = \sum_{i=1}^k \lambda_i S_i'$ the
  corresponding convex combination of slices. Since $S'$ has diameter
  2 and $S' \subset S$, $S$ has diameter 2.

  For the local diameter 2 property the result follows by taking $k=1$
  in the argument above. 

  For the diameter 2 property let $V$ be a
  relatively weakly open subset in $B_Y$. Find $y_0 \in V$ and
  $y_i^\ast \in Y^\ast$ such that $V_\varepsilon=\{y \in B_Y:
  |y_i^\ast(y-y_0)|<\varepsilon,
  i=1,\cdots,n\} \subset V$. It is possible to choose a sequence
  $(y_n)$ of
  infinitely many linearly independent points in $V_\varepsilon$ and a
  similar argument as above will now finish the proof.
\end{proof}

Our next goal is to show that ai-ideals inherit the Daugavet
property. Let us first recall the definition of this property:

\begin{defn}\label{defn:daugavet} A Banach space $X$ has the \emph{Daugavet
  property} if, for every rank 1 operator $T:X\to X$,
\[\|T+I\|=1+\|T\|.\]
\end{defn}

In Definition \ref{defn:daugavet}, $I$ denotes the identity operator
on $X$. We will need a fundamental observation, from \cite[Lemma~2.2]{KSSW}:

\begin{lem}\label{lem:fraKadSSW}
  The following are equivalent.
  \begin{enumerate}
  \item [(i)] $X$ has the Daugavet property.
  \item [(ii)] For all $y \in S_X$, $x^* \in S_{X^*}$ and $\varepsilon > 0$
    there exists $x \in S_X$ such that $x^*(x) \ge 1-\varepsilon$
    and $\|x+y\| \ge 2-\varepsilon$.
  \end{enumerate}
\end{lem}

The next result is proved for M-ideals in \cite[Proposition 2.10]{KSSW}.

\begin{prop}\label{prop:daugavet}
  If $X$ is an ai-ideal in $Y$ and
  $Y$ has the Daugavet property, then $X$
  has the Daugavet property.
\end{prop}

\begin{proof}
  Let $\varphi : \Xs \to \Ys$ be the (ai-) Hahn-Banach
  extension operator.

  We will show that (ii) of Lemma \ref{lem:fraKadSSW} is
  fulfilled. For this, let $y \in S_X$, $\xs \in S_{\Xs}$ and $\varepsilon > 0$.
  Consider the slice
  \begin{equation*}
    S_1 = \{x \in B_X : \xs(x) \ge 1- \varepsilon\}.
  \end{equation*}
  We will need to produce some $x\in S_1$ with $\|x\|=1$ and
  $\|y+x\|\geq 2-\eps$. Look at
  \begin{equation*}
    S = \{z \in B_Y : \varphi(\xs)(z) \ge 1- \eta\}.
  \end{equation*}
  Since $Y$ has the Daugavet property, for all $\eta>0$, there is some
  $z \in S$ with $\|z\|=1$ and
  such that
  $\|z+y\| \ge 2- \eta$.  Let $\frac{\eps}{2}>\eta>0$ and choose
  $\frac{\varepsilon}{2} > \delta > 0$ such that
  $\delta \le
  \frac{\frac{\varepsilon}{2}-\eta}{2-\frac{\varepsilon}{2}}$. Note
  that this choice gives $(1+\delta)^{-1}(2-\eta) \ge 2-\frac{\varepsilon}{2}$.

  Let $E = \linspan\{z,y\} \subseteq Y$, $F = \linspan\{x^\ast\} \subseteq
  X^\ast$ and find a corresponding $\delta$-isometry $T: E \to X$.
  Let $x = \frac{T(z)}{\|T(z)\|}$. Clearly $\|x\|=1$.
  We get 
  \begin{equation*}
    \|x - T(z)\|
    = |\|T(z)\|-1| \le \delta \le \frac{\varepsilon}{2},
  \end{equation*}
  hence
  \begin{equation*}
    \|x + y\|
    \ge \|T(z) + y\| - \|x - T(z)\|
    \ge (1+\delta)^{-1}(2-\eta) - \delta \ge 2-\varepsilon.
  \end{equation*}
 Finally,
  \begin{equation*}
    x^*(x) = x^*(T(z)) + x^*(x - T(z))
    \ge \varphi(x^*)(z) - \delta\ge 1-\eta -\delta \ge 1-\varepsilon,
  \end{equation*}
  and we conclude that $X$ has the Daugavet property. 
\end{proof}

\begin{rem} As in the proof of Proposition~\ref{prop:diam2subspace} the
  full strength of an ai-ideal was not needed in the above proof.
  We only needed $E$ 2-dimensional and $F$ 1-dimensional.
  Of course, as for the diameter 2 properties, we also get from
  Proposition~\ref{prop:daugavet} that $X$ inherits the Daugavet property from
  $\Xss$, but this is trivial since, from the definition of the
  Daugavet property, $X$ always has the Daugavet property if $\Xs$ has. 
\end{rem}

Recall Milne's theorem that every Banach space is a 1-complemented
subspace of a uniform algebra. Wojtasczyk observed in \cite[Corollary~4]{Woj92}
that the standard proof of this theorem yields a uniform algebra with
the Daugavet property, hence also the strong diameter 2 property. In
particular, diameter 2 properties does not automatically
pass to 1-complemented subspaces and hence not to ideals.

\section{Gurariy-spaces in terms of ai-ideals}\label{sec4}

Recall that a \emph{Lindenstrauss space} is a
Banach space such that the dual is an $L_1(\mu)$-space for some
(positive) measure $\mu$. Fakhoury \cite[Proposition~3.4]{Fak}
has proved the following result.
 
\begin{thm}\label{thm:lindenstrauss} For a Banach space $X$ the
  following statements are equivalent:
\begin{itemize}
  \item [(i)] $X$ is a Lindenstrauss space.
  \item [(ii)] $X$ is an ideal in every superspace.
\end{itemize}
\end{thm}

Below we will prove an analogous result for ai-ideals. For this we
will need the definition of a Gurariy space.  

\begin{defn}\label{defn:gurarii} A Banach space $X$ is called a
  \emph{Gurariy space} if it has the property that whenever $\eps>0$, $E$ is
  a finite-dimensional Banach space, $T_E:E\to X$ is isometric and $F$
  is a finite-dimensional Banach space with $E \subset F$, then there
  exists a linear operator $T_F:F\to X$ such that
\begin{itemize}
  \item [(i)] $T_F(f)=T_E(f)$ for all $f\in E$, and
  \item [(ii)] $(1+\eps)^{-1}\|f\|\leq\|T_F f\|\leq (1+\eps)\|f\|$ for
    all $f\in F$.
\end{itemize}
  If $T_F: F \to X$ may be taken to be isometric, then
  $X$ is called a \emph{strong Gurariy space}.  
\end{defn}

Gurariy proved in \cite{Gur} that Gurariy spaces
exist. Indeed, he constructed a separable such Banach space and showed
that all separable Gurariy spaces are linearly almost
isometric. Later Lusky \cite{Lus} proved that all separable
Gurariy spaces are in fact linearly isometric. The fact that
strong Gurariy spaces exist can be found in \cite{GarKu}.

Let us now state and prove a result similar to Theorem
\ref{thm:lindenstrauss} for ai-ideals.

\begin{thm}\label{thm:gurarii} For a Banach space $X$ the following
  statements are equivalent:
\begin{itemize}
  \item [(i)] $X$ is a Gurariy space.
  \item [(ii)] $X$ is an ai-ideal in every superspace.
\end{itemize}
\end{thm}

\begin{proof}  
  (i) $\Rightarrow$ (ii). Let $X$ be a subspace of $Y$, $E$ a
  finite-dimensional subspace of $Y$, and $\eps > 0$.  If $E \cap X$
  is of dimension $\ge 1$, then let $T: E \cap X \to X$ be the
  identity operator. By 
  assumption there is a linear extension $\hat T$ of $T$ satisfying
  $(1 + \eps)^{-1} \|e\| \le \|\hat{T}e\| \le (1 + \eps)\|e\|$ for
  every $e \in E$, just as needed. Now, if $E \cap X = \{0\}$, then choose
  some non-zero $x \in X$, put $E' = \spann(E, \{x\})$ and argue as
  above.

  (ii) $\Rightarrow$ (i). Let $\eps > 0$ and choose
  $\delta > 0$ such that $(1 + \delta)^2 \le 1 +\eps$. By \cite[Theorem
  3.6]{GarKu} we can assume $X$ is a subspace of a Gurariy space
  $X_G$. Now, let $E \subset F$ be finite-dimensional subspaces, and
  $T: E \to X$ linear and isometric. Since $X_G$ is a
  Gurariy space, there exists a linear extension $\hat{T}: F \to
  X_G$ of $T$ with $(1 + \delta)^{-1}\|f\| \le \|T(f)\| \le (1 +
  \delta)\|f\|$ for every $f \in
  F$. Put $H = \hat{T}(F)$. Since $X$ is an ai-ideal in
  $X_G$, there
  exists an operator $S: H \to X$ satisfying $(1 + \delta)^{-1}\|h\|
  \le \|Sh\| \le (1 + \delta)\|h\|$ for every $h \in H$ such
  that $Sh = h$ for every $h \in H \cap 
  X$. It follows that $S \circ \hat{T}: F \to X$ is a linear extension
  of $T$ satisfying $(1 + \eps)^{-1}\|f\|
  \le \|S \circ \hat{T}(f)\| \le (1 + \eps)\|f\|$ for every $f \in F$,
  which is exactly as desired.
\end{proof}

\begin{rem}
  It follows from the techniques used in \cite{GarKu} that every
  non-separable Banach space can be isometrically embedded in a strong
  Gurariy space. Thus by arguing as in Theorem \ref{thm:gurarii}
  it is easily seen that strong Gurariy spaces are exactly the
  spaces that are ai-ideals with $\eps = 0$ in every superspace.  
\end{rem}

\begin{cor}\label{cor:gurarii1} The separable Gurariy space is
  the only separable Banach space that is an ai-ideal in
  every superspace.
\end{cor}

The Daugavet property for Lindenstrauss spaces was studied by Werner
(see \cite[Theorem~3.5]{Wer97}). We note the following.

\begin{cor}\label{cor:gurarii2} Gurariy spaces
  enjoy the Daugavet property and hence the strong diameter 2 property.
\end{cor}

\begin{proof}
  Let $X$ be a Gurariy space.
  $C(B_{\Xs},\mbox{weak}^\ast)$ has the
  Daugavet property (cf. e.g. \cite{W}),
  and since $X$ embeds isometrically into
  $C(B_{\Xs}, \mbox{weak}^\ast)$ the result follows Theorem
  \ref{thm:gurarii} and Proposition \ref{prop:daugavet}. 
\end{proof}

It is clear from Theorems \ref{thm:lindenstrauss} and
\ref{thm:gurarii} that a Gurariy space is a Lindenstrauss
space. Thus, the last part of Corollary
\ref{cor:gurarii2} is also a particular case of the following
result (see also \cite[Corollary 3.6]{ABL}).

\begin{prop} The bidual of every infinite-dimensional
  Lindenstrauss space has the strong diameter 2 property. In
  particular every Lindenstrauss space has the strong diameter 2 property. 
\end{prop}

\begin{proof}
  Let $X$ be a Lindenstrauss space.
  It is classical that $X^*$ is order isometric
  to an $\ell_1$-sum of $L_1(\mu_a)$-spaces
  where $\mu_a$ is a probability measure  (see
  e.g. \cite[Theorem~1.b.2]{LiTz1}).
  
  Now there are two possibilities. Either
  every $\mu_a$ is purely atomic, and then
  $X^*$ is isometric to $\ell_1(\Gamma)$
  for some set $\Gamma$, or
  one $\mu_a$ is not purely atomic (see \cite[Theorem~5.14.9]{Lacey}
  for a concrete representation). In the first case $\Xss =
  \ell_\infty(\Gamma)$ which has the strong diameter 2 property.
  In the latter case we may write
  $X^{**} = Z \oplus_{\infty} L_\infty(\mu_a)$
  and thus $X^{**}$ has the strong diameter 2
  property by Proposition~4.6 in \cite{ALN}.
  In either case $X^{**}$ has the
  strong diameter 2 property.
\end{proof}

Note that not all Lindenstrauss spaces, e.g. $c_0$,
have the Daugavet property, see also \cite[p. ~79]{W}.

\section*{Acknowledgements}
We would like to thank Wies{\l}aw Kubi{\'s} for fruitful discussions
on the topic of Section \ref{sec4}. 


\providecommand{\bysame}{\leavevmode\hbox to3em{\hrulefill}\thinspace}
\providecommand{\MR}{\relax\ifhmode\unskip\space\fi MR }
\providecommand{\MRhref}[2]{%
  \href{http://www.ams.org/mathscinet-getitem?mr=#1}{#2}
}
\providecommand{\href}[2]{#2}

\end{document}